\documentclass[12pt, a4paper,reqno]{amsart}
\usepackage{amsthm}
{\theoremstyle{definition}
\newtheorem{dfn}{Definition}[section]}
\newtheorem{prop}[dfn]{Proposition}

\newtheorem{thm}[dfn]{Theorem}
{\theoremstyle{definition}
\newtheorem{rem}[dfn]{Remark}}
\newtheorem{lem}[dfn]{Lemma}
\newtheorem{cor}[dfn]{Corollary}
\newtheorem{conj}[dfn]{Conjecture}

{\theoremstyle{definition}
\newtheorem{exa}[dfn]{Example}}

\usepackage[top=30truemm,bottom=30truemm,left=35truemm,right=35truemm]{geometry}
\usepackage{amsmath,amssymb,amscd,bm,graphicx,tikz-cd,upgreek,dsfont,amsfonts,tikz,enumitem,euscript,stmaryrd}
\usetikzlibrary{cd}
\usepackage[all]{xy}
\usepackage[colorlinks=true,
    citecolor=green(pigment),
    linkcolor=alizarin,
    urlcolor=denim]{hyperref}
    
\definecolor{alizarin}{rgb}{0.82, 0.1, 0.26}
\definecolor{azure(colorwheel)}{rgb}{0.0, 0.5, 1.0}
\definecolor{blue(pigment)}{rgb}{0.2, 0.2, 0.6}
\definecolor{denim}{rgb}{0.08, 0.38, 0.74}
\definecolor{mint}{rgb}{0.24, 0.71, 0.54}
\definecolor{parisgreen}{rgb}{0.31, 0.78, 0.47}
\definecolor{persiangreen}{rgb}{0.0, 0.65, 0.58}
\definecolor{seagreen}{rgb}{0.18, 0.55, 0.34}
\definecolor{shamrockgreen}{rgb}{0.0, 0.62, 0.38}
\definecolor{green(pigment)}{rgb}{0.0, 0.65, 0.31}

\usepackage{setspace}
\newcommand{\marginparstretch}{0.6}
\let\oldmarginpar\marginpar
\renewcommand\marginpar[1]{\-\oldmarginpar[\framebox{\setstretch{\marginparstretch}\begin{minipage}{\marginparwidth}{\raggedleft\tiny #1}\end{minipage}}]{\framebox{\setstretch{\marginparstretch}\begin{minipage}{\marginparwidth}{\raggedright\tiny #1}\end{minipage}}}}

\usepackage{scalerel,stackengine}
\stackMath
\newcommand\reallywidecheck[1]{%
\savestack{\tmpbox}{\stretchto{%
   \scaleto{%
    \scalerel*[\widthof{\ensuremath{#1}}]{\kern-.6pt\bigwedge\kern-.6pt}%
    {\rule[-\textheight/2]{1ex}{\textheight}}
   }{\textheight}%
}{0.5ex}}%
\stackon[1pt]{#1}{\scalebox{-1}{\tmpbox}}%
}
\def\ang<#1>{\langle #1 \rangle}
\def\bigang<#1>{\left\langle #1 \right\rangle}

\numberwithin{equation}{section}

\newcommand{\Spec}{\operatorname{Spec}}
\newcommand{\Supp}{\operatorname{Supp}}
\newcommand{\Hom}{\operatorname{Hom}}
\newcommand{\Ext}{\operatorname{Ext}}

\newcommand{\RHom}{\operatorname{{\mathbf R}Hom}}

\newcommand{\Aut}{\operatorname{Aut}}

\newcommand{\Db}{\operatorname{D^b}}

\newcommand{\Perf}{\operatorname{Perf}}

\newcommand{\Eq}{\operatorname{Eq}}

\newcommand{\Th}{\mathbf{Th}}

\newcommand{\rk}{\operatorname{\sf rk}}

\newcommand{\Stab}{\operatorname{Stab}}

\def\Ser{\mathop{\sf Ser}\nolimits}
\def\FM{\mathop{\sf FM}\nolimits}

\def\pFM{\mathop{\sf pFM}\nolimits}

\newcommand{\cE}{\mathcal{E}}

\newcommand{\cH}{\mathcal{H}}

\newcommand{\cL}{\mathcal{L}}

\newcommand{\cO}{\mathcal{O}}
\newcommand{\cP}{\mathcal{P}}

\newcommand{\cT}{\mathcal{T}}

\newcommand{\scrA}{\EuScript{A}}

\newcommand{\scrC}{\EuScript{C}}

\newcommand{\scrE}{\EuScript{E}}

\newcommand{\scrI}{\EuScript{I}}

\newcommand{\scrM}{\EuScript{M}}

\newcommand{\scrP}{\EuScript{P}}
\newcommand{\scrQ}{\EuScript{Q}}

\newcommand{\scrS}{\EuScript{S}}
\newcommand{\scrT}{\EuScript{T}}
\newcommand{\scrU}{\EuScript{U}}

\newcommand{\bC}{\mathbb{C}}

\newcommand{\bR}{\mathbb{R}}

\newcommand{\bZ}{\mathbb{Z}}

\newcommand{\simto}{\xrightarrow{\sim}}

\newcommand{\Halg}{\widetilde{H}^{1,1}}
\newcommand{\NS}{\mathrm{NS}}
\newcommand{\ch}{\mathrm{ch}}
\newcommand{\Cone}{\operatorname{Cone}}

\newcommand{\spec}{\mathrm{ Spec}_{\triangle}}

\tikzset{
        DB/.style={circle,draw=black,circle,fill=white,inner sep=0pt, minimum size=4pt},
        DW/.style={circle,draw=black,fill=black,inner sep=0pt, minimum size=4pt},
        cvertex/.style={circle,draw=black,fill=white,inner sep=1pt,outer sep=3pt},
        vertex/.style={circle,fill=black,inner sep=1pt,outer sep=3pt},
        star/.style={circle,fill=yellow,inner sep=0.75pt,outer sep=0.75pt},
        tvertex/.style={inner sep=1pt,font=\scriptsize},
	pvertex/.style={circle,inner sep=1pt,outer sep=2pt,font=\scriptsize},
        gap/.style={inner sep=0.5pt,fill=white}
}

\makeatletter
\newcommand*{\defeq}{\mathrel{\rlap{%
                     \raisebox{0.3ex}{$\m@th\cdot$}}%
                     \raisebox{-0.3ex}{$\m@th\cdot$}}%
                     =}
\makeatother

\begin{document}

\title[]{Fourier-Mukai loci of K3 surfaces of Picard number one}
\author[Y.~Hirano]{Yuki Hirano}
\author[G.~Ouchi]{Genki Ouchi}

\address{Y.~Hirano, Tokyo University of Agriculture and Technology, 2–24–16 Nakacho, Koganei, Tokyo 184–8588, Japan}
\email{hirano@go.tuat.ac.jp}

\address{G.~Ouchi, Graduate School of Mathematics, Nagoya University, Furocho, Chikusaku, Nagoya, Japan, 464-8602}
\email{genki.ouchi@math.nagoya-u.ac.jp}

\begin{abstract}
In this paper, we describe the Fourier-Mukai locus of the derived category of a complex algebraic K3 surface of Picard number one. We also prove that the Fourier-Mukai locus of the derived category of a complex algebraic K3 surface of Picard number one is strictly smaller than it's Matsui spectrum.
\end{abstract}

\maketitle{}

\section{Introduction}

\subsection{Balmer spectrum}~

For a tensor triangulated category $(\scrT,\otimes)$, Balmer \cite{balmer2,balmer,balmer3}
constructed the locally ringed space $(\Spec_{\otimes}\scrT, \cO_{\scrT, \otimes})$, which is called the {\it Balmer spectrum} of $(\scrT, \otimes)$. It is known that the set $\Th(\scrT)$ of all thick subcategories of a triangulated category $\scrT$ has the topology called the Zariski topology. The topological space $\Spec_{\otimes}\scrT$ is constructed as the topological subspace of $\Th(\scrT)$ that consists of prime thick ideals of the tensor triangulated category $(\scrT,\otimes)$.
The {\it tensor triangular geometry} \cite{balmer4} is the study of tensor triangulated categories by the Balmer spectra, and it is developed in various areas including algebraic geometry, stable homotopy theory, modular representation theory and motivic theory. In algebraic geometry, it is known that the Balmer spectrum of the tensor triangulated category $(\Perf X, \otimes_X)$ associated to a noetherian scheme $X$ is isomorphic to the scheme $X$ as ringed spaces, where $\Perf X$ is the perfect derived category of $X$ and $\otimes_X$ is the derived tensor product on $X$.

\subsection{Matsui spectrum}~

For a triangulated category $\scrT$, Matsui \cite{matsui,matsui2} constructed the ringed space $(\Spec_{\triangle}\scrT, \cO_{\scrT})$, which is called the {\it Matsui spectrum} of $\scrT$. The topological space $\Spec_{\triangle}\scrT$ is constructed as the topological subspace of $\Th(\scrT)$ that consists of prime thick subcategories of $\scrT$. For a noetherian reduced quasi-affine scheme $X$, the Matsui spectrum of $\Perf X$ is isomorphic to $X$ as ringed spaces. In general, 
for a noetherian scheme $X$, the Balmer spectrum $\Spec_{\otimes_X}\Perf X$ is a topological subspace of the Matsui spectrum $\Spec_{\triangle}\Perf X$. 
For noetherian schemes $X$ and $Y$, if $\Phi \colon \Perf X \simto \Perf Y$ is an equivalence, $\Phi$ induces the embedding $\Phi \colon \Spec_{\otimes_Y}\Perf Y \to \Spec_{\triangle}\Perf X$ as topological spaces. In particular, for a smooth variety $X$, the Matsui spectrum $\Spec_{\triangle}\Db(X)$ of the derived category $\Db(X)$ of $X$ contains all Fourier-Mukai partners as Zariski topological spaces. Therefore, it is interesting to determine the Matsui spectrum of the derived category of a smooth variety. 
The following are the only known cases other than smooth quasi-affine varieties. In the rest of this section, we work over an algebraically closed field $k$.

\begin{exa}[\cite{ho}, \cite{matsui2}]\label{ex:known case}
Let $X$ be a smooth projective curve of genus $g$.
\begin{itemize}
\item[(1)] If $g=0$, we have an isomorphism
\[ \Spec_{\triangle}\Db(X) \simeq X \sqcup \bZ \]
as ringed spaces, where $X$ is realized as  $\Spec_{\otimes_X}\Db(X)$.
\item[(2)] Consider the set 
\[I\defeq\{(r,d)\in \bZ^2 \mid r>0, \gcd(r,d)=1\}.\]
If $g=1$, we have an isomorphism
\[\spec \Db(X) \simeq X \sqcup \bigsqcup_{(r,d) \in I} M(r,d) \]
as ringed spaces, where $X$ is realized as $\Spec_{\otimes_X}\Db(X)$ of $\Db(X)$ and $M(r,d)$ is the moduli space of $\mu$-semistable vector bundles on $X$ with the Chern character $(r,d)$. Note that $M(r,d)$ is also an elliptic curve isomorphic to $X$.
\end{itemize}
\end{exa}

In general, it is not easy to determine the Matsui spectrum $\Spec_{\triangle}\Db(X)$ of the derived category $\Db(X)$ of a smooth projective variety $X$. For example, even if $X$ is a smooth projective curve of genus $g \geq 2$, we do not have a complete description of the topological space $\Spec_{\triangle}\Db(X)$. 

\begin{rem}[Proposition \ref{prop:high genus}]
For a smooth projective curve $X$ of genus $g \geq 2$, we have
\[ \Spec_{\triangle}\Db(X) \setminus \Spec_{\otimes_X}\Db(X) \neq \emptyset. \]
\end{rem}

\subsection{Serre invariant locus and Fourier-Mukai locus}~

Let $X$ be a smooth projective variety.
Instead of studying the whole space $\Spec_{\triangle}\Db(X)$, we often study simpler subspaces of $\Spec_{\triangle}\Db(X)$. By the same construction as in \cite{matsui2}, for any topological subspace $S$ of $\Th(\scrT)$, we can regard $S$ as a ringed space $(S,\cO_S)$.  The {\it Serre invariant locus} $\Spec^{\Ser}_{\triangle}\Db(X)$, introduced in \cite{ho}, consists of prime thick subcategories of $\Db(X)$ invariant under the Serre functor. If $X$ is a Calabi-Yau variety or a Fano variety, we have the following.

\begin{rem}[{\cite[Corollary 5.4]{ho}}]\label{rem:CY and Fano, intro}
The following holds. 
\begin{itemize}
\item[(1)] Assume that $X$ is a Calabi-Yau variety, that is $\omega_X \simeq \cO_X$.
Then we have
\[ \Spec_{\triangle}^{\Ser}\Db(X)=\Spec_{\triangle}\Db(X) \]
as ringed spaces.
\item[(2)] Assume that $\omega_X$ or $\omega^{-1}_X$ is ample. Then we have
\[ \Spec_{\triangle}^{\Ser}\Db(X)=\Spec_{\otimes_X}\Db(X) \]
as ringed spaces. This reproves Bondal-Orlov's reconstruction theorem in \cite{bo}.
\end{itemize}
\end{rem}

In \cite{ito}, Ito introduced the {\it Fourier-Mukai locus} $\Spec^{\FM}_{\triangle}\Db(X)$ of $\Db(X)$. The Fourier-Mukai locus $\Spec^{\FM}_{\triangle}\Db(X)$ consists of prime thick subcategories corresponding to prime thick ideals of Fourier-Mukai partners of $X$. 
Then there are inclusions
\[ \Spec_{\otimes_X}\Db(X) \subset \Spec^{\FM}_{\triangle}\Db(X) \subset \Spec^{\Ser}_{\triangle}\Db(X) \subset \Spec_{\triangle}\Db(X). \]
If $X$ is a smooth projective curve or a smooth projective variety with the ample (anti-)canonical line bundle, the Serre invariant locus $\Spec_{\triangle}^{\Ser}\Db(X)$ and the Fourier-Mukai locus $\Spec^{\FM}_{\triangle}\Db(X)$ coincide as ringed spaces. In \cite{ito}, Ito proposed the following conjecture.

\begin{conj}[{\cite[Conjecture 6.14]{ito}}]\label{conj:Ito, intro}
For a smooth projective variety $X$,
\[ \Spec^{\FM}_{\triangle}\Db(X)=\Spec^{\Ser}_{\triangle}\Db(X) \]
holds.
\end{conj}

\subsection{Results}~

In this paper, we study the Serre invariant locus and the Fourier-Mukai locus of the derived category of a complex algebraic K3 surface of Picard number one. 
Let $X$ be a complex algebraic K3 surface. 
An autoequivalence $\Phi \colon \Db(X) \simto \Db(X)$ is {\it standard} if $\Phi$ is isomorphic to a composition of shifts, automorphisms of $X$ and tensoring line bundles on $X$. 
Denote by $\Aut^{\mathrm{st}}\Db(X)$  the subgroup of standard autoequivalences of $\Db(X)$, and 
write $\FM(X)$ for the set of isomorphism classes of Fourier-Mukai partners of $X$.
Take a Fourier-Mukai partner $Y$ of $X$, and write $\Eq(\Db(Y), \Db(X))$ for the set of isomorphism classes of equivalences from $\Db(Y)$ to $\Db(X)$. Then $\Aut\Db(Y)$ acts on $\Eq(\Db(Y), \Db(X))$ from the right. Consider the set
\[ I_Y\defeq\Eq(\Db(Y), \Db(X))/\Aut^{\mathrm{st}}\Db(Y). \]
The following is the first main theorem of this paper.

\begin{thm}[Theorem \ref{thm:main A}]\label{thm:main 1}
Assume that the Picard number of $X$ is one. 
Then we have
\[ \Spec^{\FM}_{\triangle}\Db(X)=\bigsqcup_{Y \in \FM(X)} \bigsqcup_{\Phi \in I_Y}\Phi\left(\Spec_{\otimes_Y}\Db(Y) \right). \]
\end{thm}

The following is the second main theorem, which gives a counterexample of Conjecture
\ref{conj:Ito, intro}. 

\begin{thm}[Theorem \ref{thm:main B}]\label{thm:main 2}
Assume that the Picard number of $X$ is one.
Then we have
\[ \Spec^{\Ser}_{\triangle}\Db(X) \setminus \Spec^{\FM}_{\triangle}\Db(X) \neq \emptyset. \]
\end{thm}



\subsection{Notation and Convention}
\begin{itemize}
    \item In this paper, we treat only essentially small triangulated categories over an algebraically closed field $k$.
    \item  For a noetherian scheme $X$, denote the bounded derived category of coherent sheaves on $X$ (resp. the derived category of perfect complexes on $X$) by $\Db(X)$ (resp. $\Perf X$). 
    \item We denote by $\otimes_X$ the derived tensor product on $X$.
    \item For a proper morphism $f\colon X\to Y$ of smooth varieties, we write $f_*\colon \Db(X)\to \Db(Y)$ and  $f^*\colon \Db(Y)\to \Db(X)$ for the derived push-forward and the derived pull-back respectively. 
    \item For a variety (or a scheme) $X$,  a point in $X$ is not necessarily a closed point. 
    \item We treat only algebraic K3 surfaces.
\end{itemize}

\subsection{Acknowledgements}~

  Y.H. is supported by JSPS KAKENHI Grant Number 23K12956. G.O. is supported by JSPS KAKENHI Grant Number 19K14520.

\section{Spectra of triangulated categories}

\subsection{Thick subcategories}~

In this subsection, we recall the topology on the set of all thick subcategories of a triangulated category. 

Let $\scrT$ be a triangulated category. 
A {\it thick subcategory} of $\scrT$ is a full triangulated subcategory of $\scrT$ that is closed under taking direct summands.
Denote by $\Th(\scrT)$  the set of all thick subcategories of $\scrT$.
For a full subcategory $\scrE$ of $\scrT$, we define
\[ Z(\scrE)\defeq\{\scrC \in \Th(\scrT) \mid \scrC \cap \scrE=\emptyset \}. \]
Then it is easy to see the following (see \cite[Definition 2.1]{mt}). 
\begin{itemize}
    \item[$(1)$]We have  $Z(\scrT)=\emptyset$ and  $Z(\emptyset)=\Th(\scrT)$.
    \item[$(2)$]For a family $\{\scrE_i\}_{i \in I}$ of full subcategories of $\scrT$, we have  $\bigcap_{i \in I}Z(\scrE_i)=Z(\bigcup_{i \in I}\scrE_i)$.
    \item[$(3)$]For full subcategories $\scrE_1, \cE_2$ of $\scrT$, we have
    $Z(\scrE_1) \bigcup Z(\scrE_2)=Z(\scrE_1 \oplus \scrE_2)$, where $\scrE_1 \oplus \scrE_2\defeq \{E_1 \oplus E_2 \in \cT \mid E_1 \in \cE_1, E_2 \in \scrE_2 \}$.
\end{itemize}

In particular, the set $\Th(\scrT)$ has the topology whose closed subsets are of the form $Z(\scrE)$.

\subsection{Balmer spectrum}~

In this subsection, we recall Balmer spectra of tensor triangulated categories following \cite{balmer2} and \cite{balmer}. First, we recall the definition of tensor triangulated categories. 

\begin{dfn}[{\cite[Definition 1.1]{balmer2}}]
A tensor triangulated category $(\scrT, \otimes)$ is a symmetric monoidal category with the unit object $\mathbf{1}_{\scrT}$ such that $\scrT$ is a triangulated category and the bi-functor $\otimes \colon \scrT \times \scrT \to \scrT$ is exact in each variable. 
\end{dfn}

Let $(\scrT, \otimes)$ be a tensor triangulated category.
A thick subcategory $\scrI$ of $\scrT$ is called an {\it ideal} of $(\scrT, \otimes)$ if $A \otimes E \in \scrI$ holds for any objects $A \in \scrT$ and $E \in \scrI$. 
For an ideal $\scrI$ of $(\scrT, \otimes)$,  we define the {\it radical} $\sqrt{\scrI}$ of $\scrI$ as 
\[ \sqrt{\scrI}\defeq\{ E \in \scrT \mid E^{\otimes n} \in \scrI \ \text {for some positive integer $n$}  \}. \]
An ideal $\scrI$ of $(\scrT, \otimes)$ is said to be {\it radical}  if $\sqrt{\scrI}=\scrI$ holds.

\begin{dfn}[{\cite[Definition 2.1]{balmer2}}]
An ideal $\scrP\neq \scrT$ of $(\scrT, \otimes)$ is said to be  {\it prime}  if for objects $A,B \in \scrT$, the condition $A \otimes B \in \scrP$ implies $A\in \scrP$ or $B \in \scrP$.
The set of all prime ideals of $(\scrT,\otimes)$ is denoted by 
\[
\Spec_{\otimes}\scrT.
\]  
We regard $\Spec_\otimes\scrT$ as a topological subspace of $\Th(\scrT)$. The topological space $\Spec_\otimes\scrT$ is called the {\it Balmer spectrum} of $(\scrT, \otimes)$.
\end{dfn}

Let $X$ be a noetherian scheme. For an object $E \in\Perf X$, we define the {\it support} $\Supp E$ of the object $E$ as
\[
\Supp E\defeq\{x\in X\mid E_x\not\cong 0 \mbox{ in }\Perf\cO_{X,x}\}.
\]

Note that 
$\Supp E=\bigcup_{i\in \bZ}\Supp \cH^i(E)$. Since the support  of each sheaf cohomology $\cH^i(E)$  is closed in $X$, and $\cH^i(E)\neq0$ only for finitely many $i\in\bZ$, the support $\Supp E$ is closed in $X$.

\begin{dfn}
For $x \in X$, we define the ideal $\scrS_X(x)$ of $(\Perf X, \otimes_X)$ as
\[
\scrS_X(x)\defeq\{E \in \Perf X \mid x\notin \Supp(E)\}.
\]
\end{dfn}

The following property is important in the proof of Theorem \ref{thm:main 2}.

\begin{rem}\label{rem:torsion}
Take $x \in X$.
For any object $E \in \scrS_X(x)$, we have $\rk E=0$, where $\rk E\defeq \sum_{i\in \bZ}(-1)^i\rk\cH^i(E)$.
\end{rem}

In \cite{balmer}, Balmer reconstructs a noetherian scheme $X$ from the tensor triangulated category $(\Perf X,\otimes_{X})$. If no confusion can arise, we simply write $\otimes$ for $\otimes_X$.

\begin{thm}[{\cite[Theorem 6.3 (a)]{balmer}}]\label{thm:scheme vs Balmer}
Let $X$ be a noetherian scheme. There is a homeomorphism
\[ \scrS_X \colon X \simto \Spec_\otimes \Perf X,~x \mapsto \scrS_X(x) \]
of ringed spaces.
\end{thm}

\subsection{Matsui spectrum}~

In this subsection, we recall the definition of Matsui spectra following \cite{matsui} and \cite{matsui2}. 
Let $\scrT$ be a triangulated category.

\begin{dfn}[{\cite[Definition 2.2]{matsui}},{\cite[Definition 2.8]{matsui2}}]\label{def:prime thick subcategory}
A thick subcategory $\scrP$ of $\scrT$ is said to be {\it prime} if there exists the smallest element $\overline{\scrP}$ in the partially orderd set $ \{\scrQ\in\Th(\scrT)\mid \scrP\subsetneq\scrQ\}$. We denote by  
\[
\Spec_{\triangle}\scrT
\] 
the set of prime thick subcategories of $\scrT$. The set $\Spec_{\triangle}\scrT$ has the induced topology from $\Th(\scrT)$. The topological space $\Spec_{\triangle}\scrT$ is called the {\it Matsui spectrum} of $\scrT$.
\end{dfn}

Definition \ref{def:prime thick subcategory} gives a generalization of prime ideals of a tensor triangulated category of the form $(\Perf X,\otimes_X)$, where $X$ is a noetherian scheme.  

\begin{prop}[{\cite[Theorem 1.4]{matsui}}]\label{prop:Matusi vs Balmer}
Let $(\scrT, \otimes)$ be a tensor triangulated category.
For a radical ideal $\scrP$ of $(\scrT, \otimes)$, if $\scrP$ is a prime thick subcategory, then $\scrP$ is a prime ideal of $(\scrT, \otimes)$. When $(\scrT, \otimes)=(\Perf X, \otimes_X)$ holds for some noetherian scheme, the converse holds. 
\end{prop}

A thick subcategory $\scrP$ of $\scrT$ is said to be {\it maximal} if  $\scrP$ is maximal in the partially ordered set $\Th(\scrT)\setminus \{ \scrT \}$.
Note that a maximal thick subcategory of $\scrT$ is prime.
An object $G$ of $\scrT$ is called a {\it split generator} of $\scrT$ if the smallest thick subcategory of $\scrT$ containing $G$ is equal to $\scrT$. 
For the non-emptyness of $\Spec_{\triangle}\scrT$, the following holds.

\begin{prop}[{\cite[Proposition 2.15]{ho2}}]\label{prop:non-empty}
Assume that $\scrT$ has a split generator $G$.
Then for any thick subcategory $\scrU\subsetneq\scrT$, there is a maximal thick subcategory of $\scrT$ containing $\scrU$.
In particular, the Matsui spectrum $\Spec_{\triangle}\scrT$ of $\scrT$ is non-empty.
\end{prop}

\begin{exa}[{\cite[Theorem 4]{orlov}}]\label{ex:split generator}
Let $X$ be a smooth projective variety of dimension $d$. 
Take a very ample line bundle $\cO_X(1)$ on $X$. 
Then the object
\[ G\defeq\oplus_{i=0}^{d}\cO_X(i)\]
is a split generator of $\Db(X)$.

For an admissible subcategory $\scrA$ of $\Db(X)$ with the inclusion functor $\iota \colon \scrA \to \Db(X)$, the object $\iota^*(G)$ is a split generator of $\scrA$, where $\iota^* \colon \Db(X) \to \scrA$ is the left adjoint functor of $\iota$.
\end{exa}

By Proposition \ref{prop:non-empty} and Example \ref{ex:split generator}, we have the following statement.

\begin{cor}[{\cite[Corollary 2.16]{ho2}}]\label{cor:empty}
Let $X$ be a smooth projective variety.
For a non-zero admissible subcategory $\scrA$ of $\Db(X)$, there is a maximal thick subcategory of $\scrA$. In particular, the Matsui spectrum $\Spec_{\triangle}\scrA$ of $\scrA$ is non-empty.
\end{cor}

\subsection{Serre invariant locus}~

In this subsection, we introduce the Serre invariant locus of Matsui spectrum of a triangulated category following \cite{ho}.

Let $\cT$ be a triangulated category with finite dimensional Hom-spaces. 
A {\it Serre functor} $S\colon \scrT \simto \scrT$ of $\scrT$ is an exact autoequivalence such that for any objects $A,B\in \scrT$, there is a functorial isomorphism
\[
\Hom_{\scrT}(A,B)\cong \Hom_{\scrT}(B,S(A))^*,
\]
of vector spaces, where $(-)^*$ denotes the dual of a vector space over $k$. It is standard that  Serre functors, if they exist, are unique up to functor isomorphisms (see \cite[Section 1.1]{huybre}).  

\begin{dfn}[{\cite[Definition 5.1]{ho}}]\label{def:serre}
Assume that $\scrT$ admits a Serre functor $S \colon \scrT \simto \scrT$.
A prime thick subcategory $\scrP$ of $\scrT$ is said to be {\it Serre invariant}, if $S(\scrP)=\scrP$.  We denote by 
\[
\Spec_{\triangle}^{\Ser}\scrT
\]
the set of all Serre invariant prime thick subcategories of $\scrT$.
\end{dfn}


\begin{rem}[{\cite[Corollary 5.4]{ho}}]\label{rem:CY and Serre}
Let $X$ be a smooth projective variety.
\begin{itemize}
\item[(1)] Assume that $X$ is a Calabi-Yau variety, that is $\omega_X \simeq \cO_X$. Then the shift functor 
\[ [\dim X] \colon \Db(X) \simto \Db(X)\]
is a Serre functor of $\Db(X)$.  We have
\[ \Spec_{\triangle}^{\Ser}\Db(X)=\Spec_{\triangle}\Db(X). \]
\item[(2)] Assume that $\omega_X$ or $\omega^{-1}_X$ is ample. Then we have
\[ \Spec_{\triangle}^{\Ser}\Db(X)=\Spec_{\otimes}\Db(X). \]
\end{itemize}
\end{rem}

In Remark \ref{rem:CY and Serre} (2), the Serre invariant locus is strictly smaller than the Matsui spectrum in general, e.g., Example \ref{ex:known case} (1).  

\begin{prop}\label{prop:high genus}
Let $X$ be a smooth projective curve of genus $g \geq 2$.
Then we have
\[ \Spec_{\triangle}\Db(X) \setminus \Spec^{\Ser}_{\triangle}\Db(X) \neq \emptyset. \]
\end{prop}
\begin{proof}
 Note that $\cO_X$ is vertex like in the sense of \cite[Definition 3.7]{el}. By \cite[Proposition 3.9]{el}, $\cO_X$ is not a split generator of $\Db(X)$. By Proposition \ref{prop:non-empty}, there is a maximal thick  subcategory $\scrM$ containing $\cO_X$. Since $\scrM$ is torsion-free, it is not an ideal of $(\Db(X),\otimes)$ by \cite[Corollary 4.6]{ho}. Hence $\scrM\not\in\Spec_{\otimes}\Db(X)=\Spec_{\triangle}^{\Ser}\Db(X)$ by Remark \ref{rem:CY and Serre}.
\end{proof}

\subsection{Fourier-Mukai locus}~

In this subsection, we introduce the Fourier-Mukai locus of a triangulated category following \cite{ito}.
Let $\scrT$ be a triangulated category.

\begin{dfn}\label{def:Fourier-Mukai partner}
Denote by $\FM(\scrT)$  the set of isomorphism classes of a smooth projective variety $X$ with $\scrT \simeq \Db(X)$.
Similarly, we write $\pFM(\scrT)$ for the set of isomorphism classes of a smooth proper variety $X$ with $\scrT \simeq \Db(X)$. A variety $X$ is a {\it Fourier-Mukai partner} of $\scrT$ if $X \in \FM(\scrT)$ holds. Similarly, a variety $X$ is a {\it proper Fourier-Mukai partner} of $\scrT$ if $X \in \pFM(\scrT)$ holds.
\end{dfn}

For $\pFM(\scrT)$ in Definition \ref{def:Fourier-Mukai partner}, see {\cite[Observation 6.16]{ito}}. 
For a smooth proper variety $X$, we consider the set $\Eq(\Db(X), \scrT)$ of isomorphism classes of exact equivalences from $\Db(X)$ to $\scrT$.

\begin{dfn}[{\cite[Construction 3.5]{ito}}]\label{def:tt-spectrum of FM}
Let $X$ be a smooth proper variety with $\scrT \simeq \Db(X)$. 
We define the topological subspace $\Spec_{X}\scrT$ of $\Spec_{\triangle}\scrT$ as 
\[ \Spec_{X}\scrT\defeq\bigcup_{\Phi \in \Eq(\Db(X), \scrT)}\Phi\left(\Spec_{\otimes}\Db(X)\right). \]
The topological subspace $\Spec_{X}\scrT $ is called the {\it tensor triangulated spectrum of a Fourier-Mukai partner} $X$ of $\scrT$. 
\end{dfn}

Note that the autoequivalence group $\Aut\scrT$ acts on $\Spec_{X}\scrT$ for $X \in \pFM(\scrT)$. Ito \cite{ito} introduced the following topological subspace of $\Spec_{\triangle}(\scrT)$.

\begin{dfn}[{\cite[Definition 4.1]{ito}}]\label{def:FM locus}
We define the {\it Fourier-Mukai locus} $\Spec^{\FM}_{\triangle}\scrT$ of $\scrT$ as
\[ \Spec^{\FM}_{\triangle}\scrT\defeq\bigcup_{X \in \FM(\scrT)}\Spec_{X}\scrT.  \]
\end{dfn}

It is natural to consider the proper version of Definition \ref{def:FM locus}.

\begin{dfn}[{\cite[Observation 6.16]{ito}}]\label{def:pFM}
We define the {\it proper Fourier-Mukai locus} $\Spec^{\pFM}_{\triangle}\scrT$ of $\scrT$ as
\[ \Spec^{\pFM}_{\triangle}\scrT\defeq\bigcup_{X \in \pFM(\scrT)}\Spec_{X}\scrT.  \]
\end{dfn}
Note that the autoequivalence group $\Aut\scrT$ of $\scrT$ acts on $\Spec^{\FM}_{\triangle}\scrT$ and $\Spec^{\FM}_{\triangle}\scrT$. By \cite[Observation 6.16]{ito}, we have
\[ \Spec^{\FM}_{\triangle}\scrT=\Spec^{\pFM}_{\triangle}\scrT \] 
if $\scrT \simeq \Db(X)$ holds for some smooth projective variety $X$ with $\dim X \leq 2$.



The following is the relation among the Serre invariant locus $\Spec^{\Ser}_{\triangle}\scrT$ and the Fourier-Mukai loci $\Spec^{\FM}_{\triangle}\scrT$ and $\Spec^{\pFM}_{\triangle}\scrT$.

\begin{rem}[{\cite[Corollary 6.3, Observation 6.16]{ito}}]
Assume that $\scrT$ admits a Serre functor.
We have the following inclusions
\[ \Spec^{\FM}_{\triangle}\scrT \subset \Spec^{\pFM}_{\triangle}\scrT \subset \Spec^{\Ser}_{\triangle}\scrT \subset \Spec_{\triangle}\scrT. \]
\end{rem}

Ito  proposed the following conjectures.

\begin{conj}[{\cite[Conjecture 6.14]{ito}}]\label{conj:strong}
Assume that $\FM(\scrT)$ is non-empty.
Then we have
\[ \Spec^{\FM}_{\triangle}\scrT=\Spec^{\Ser}_{\triangle}\scrT. \]
\end{conj}

The following is a weaker version of Conjecture \ref{conj:strong}.

\begin{conj}[{\cite[Conjecture 6.17]{ito}}]\label{conj:weak}
Assume that $\pFM(\scrT)$ is non-empty.
Then we have
\[ \Spec^{\pFM}_{\triangle}\scrT=\Spec^{\Ser}_{\triangle}\scrT. \]
\end{conj}

For some special cases, Conjecture \ref{conj:strong} is known to be true.

\begin{rem}[{\cite[Theorem 1.6]{ito}}]
Assume that $\scrT \simeq \Db(X)$ holds for some smooth projective variety $X$.
If $X$ is a curve or $X$ has the ample (anti-) canonical line bundle, 
Conjecture \ref{conj:strong} is true.
\end{rem}

We will give a counterexample of Conjecture \ref{conj:strong} and Conjecture \ref{conj:weak} using K3 surfaces later.
For non-commutative K3 surfaces, we give the following remark.

\begin{rem}
Let $X$ be a complex smooth cubic fourfold.
By \cite{kuz10}, we have the semi-orthogonal decomposition
\[ \Db(X)=\langle \scrA_X, \cO_X, \cO_X(1), \cO_X(2) \rangle  \]
and the admissible subcategory $\scrA_X$ is a $2$-dimensional Calabi-Yau category. 
Therefore, we have
\[ \Spec^{\Ser}_{\triangle}\scrA_X=\Spec_{\triangle}\scrA_X \neq \emptyset \]
by Proposition \ref{cor:empty}.

Since the Hochschild homology $HH_*(\scrA_X)$ of $\scrA_X$ is a $24$-dimensional vector space by \cite[Corollary 7.5]{kuz}, $\scrA_X$ is not equivalent to the derived category of an abelian surface.  

Assume that $X$ is very general, that is $\rk H^{2,2}(X,\bZ)=1$. Then it is well known that $\scrA_X$ is not equivalent to the derived category of a K3 surface. For example, see \cite[Theorem 1.1]{at14}. 
In particular, 
\[ \FM(\scrA_X)=\pFM(\scrA_X)=\emptyset\]
holds, and then we have
\[ \Spec^{\FM}_{\triangle}(\scrA_X)=\Spec^{\pFM}_{\triangle}(\scrA_X)=\emptyset. \]
Therefore, we obtain
\[ \Spec^{\FM}_{\triangle}(\scrA_X)=\Spec^{\pFM}_{\triangle}(\scrA_X) \subsetneq \Spec^{\Ser}_{\triangle}\scrA_X. \]
\end{rem}

\section{Birational Fourier-Mukai transforms}
\subsection{Fourier-Mukai transforms}~

In this subsection, we recall the notion of birational Fourier-Mukai transforms following \cite{voet}.

Let $X$ and $Y$ be smooth proper varieties with the projection morphisms $p \colon X \times Y \to Y$ and $q \colon X \times Y \to X$. For an object $\cE \in \Db(X \times Y)$, the exact functor
\[ \Phi_\cE\defeq p_*(q^*(-) \otimes_{X\times Y} \cE) \colon \Db(X) \to \Db(Y) \]
is called the {\it Fourier-Mukai functor} with the Fourier-Mukai kernel $\cE$. If $\Phi_\cE$ is an equivalence, $\Phi_\cE$ is called a {\it Fourier-Mukai transform}. The following is the definition of birational Fourier-Mukai transformations.

\begin{dfn}[{\cite[Definition 4.0.1]{voet}}]\label{def:birational FM transform}
Let $\Phi \colon \Db(X) \simto \Db(Y)$ be an equivalence. 
Take $\cE \in \Db(X \times Y)$ satisfying $\Phi \simeq \Phi_\cE$. 
Then $\Phi$ is {\it birational} if there is a non-empty open subset $U$ of $X$ and an open immersion $\varphi \colon U \to Y$ such that $\cE|_{U \times Y}$ is isomorphic to the structure sheaf of the graph of $\varphi$.
\end{dfn}

\begin{rem}\label{rem:point to point}
In Definition \ref{def:birational FM transform}, if $\Phi_\cE \colon \Db(X) \simto \Db(Y)$ is a birational Fourier-Mukai transform, we have $\Phi_\cE(\cO_x) \simeq \cO_{(\varphi(x))}$ holds for any closed point $x \in U$.
\end{rem}

An autoequivalence $\Phi \colon \Db(X) \simto \Db(X)$ is {\it standard} if $\Phi$ is isomorphic to a composition of shifts, automorphisms of $X$ and tensoring line bundles on $X$. 
Let $\Aut^{\mathrm{st}}\Db(X)$ be the subgroup of standard autoequivalences of $\Db(X)$.

\subsection{Intersection of Balmer spectra}~

In this subsection, we recall the result in \cite{ito} about the intersection of Balmer spectra and birational Fourier-Mukai transforms.

Let $\scrT$ be a triangulated category, and  
assume that $\FM(\scrT)$ is non-empty.
For birational Fourier-Mukai transforms, we recall the following proposition.

\begin{prop}[{\cite[Lemma 4.11]{ito}}]\label{prop:intersection}
Let $X$ and $Y$ be smooth projective varieties with equivalences
\[ \Phi_1 \colon \Db(X) \simto \scrT,~\Phi_2 \colon \Db(Y) \simto \scrT.  \]
Then 
\[ \Phi_1\left(\Spec_{\otimes}\Db(X)\right) \cap \Phi_2\left(\Spec_{\otimes}\Db(Y)\right) \neq \emptyset \]
holds if and only if $\Phi^{-1}_2 \circ \Phi_1 \colon \Db(X) \simto \Db(Y)$ is birational up to shifts.
\end{prop}

For standard autoequivalences, we have the following proposition.

\begin{prop}\label{prop:standard}
Let $X$ be a smooth projective variety with equivalences
\[ \Phi_1 \colon \Db(X) \simto \scrT,~\Phi_2 \colon \Db(X) \simto \scrT.  \]
If $\Phi^{-1}_2 \circ \Phi_1$ is a standard autoequivalence of $\Db(X)$, we have
\[ \Phi_1\left(\Spec_{\otimes_X}\Db(X)\right)=\Phi_2\left(\Spec_{\otimes_X}\Db(X)\right).  \]
\end{prop}
\begin{proof}
For an automorphism $f$ of $X$, $f^* \colon \Db(X) \simto \Db(X)$ induces the automorphism 
\[ f^* \colon \Spec_{\otimes} \Db(X) \simto \Spec_{\otimes} \Db(X). \]
For a line bundle $\cL$ on $X$ and an integer $n$, 
$(-) \otimes_X \cL[n] \colon \Db(X) \simto \Db(X)$ acts on $\Spec_{\otimes} \Db(X)$ trivially. Assume that  $\Phi\defeq\Phi^{-1}_2 \circ \Phi_1$ is a standard autoequivalence of $\Db(X)$. Then we have $\Phi_2 \circ \Phi=\Phi_1$. 
Since 
\[\Phi\left(\Spec_{\otimes}\Db(X)\right)=\Spec_{\otimes}\Db(X)\]
holds, we obtain the desired claim.
\end{proof}

\section{Bridgeland Stability condition}
\subsection{Stability conditions on K3 surfaces}~

In this section, we quickly recall the basics of stability conditions on derived categories of K3 surfaces following \cite{bri07} and \cite{bri08}.

Let $X$ be a complex K3 surface.
The {\it algebraic Mukai lattice} $\Halg(X,\bZ)$ of $X$ is defined as
\[ \Halg(X,\bZ)\defeq H^0(X,\bZ) \oplus \NS(X) \oplus H^4(X,\bZ). \]
The structure of a lattice is defined by the {\it Mukai pairing}
\[ \langle (r_1,c_1, m_1), (r_2,c_2,m_2) \rangle\defeq c_1c_2-r_1m_2-r_2m_1  \]
for $(r_1, c_1, m_1), (r_2, c_2, m_2) \in \Halg(X,\bZ)$. Then we have the surjective homomorphism 
\[ v_X \colon K_0(X) \twoheadrightarrow \Halg(X,\bZ),~[E] \mapsto \ch(E) \cdot \sqrt{\mathrm{td}_X}. \]
For an object $E \in D^b(X)$, the cohomology class $v_X(E)$ is called the {\it Mukai vector} of $E$. 
By Riemann--Roch formula, we have
\[ \langle v_X(E), v_X(F) \rangle=-\chi(E,F) \]
for any objects $E,F \in \Db(X)$, where $\chi(E,F)$ is defined by
\[ \chi(E,F) = \sum_{n \in \bZ} (-1)^n \dim \Ext^n(E,F). \]

The following is the definition of stability conditions on $\Db(X)$. Fix a norm $||-||$ on the real vector space $\Halg(X,\bZ) \otimes \bR$.

\begin{dfn}[{\cite[Definition 5.1]{bri07}}]
 A {\it stability condition} $\sigma=(Z, \cP)$ on $\Db(X)$ is a pair of a homomorphism $Z \colon \Halg(X,\bZ) \to \bC$ and a collection $\cP=\{\cP(\phi) \}_{ \phi \in \bR}$ of full additive subcategories $\cP(\phi)$ of $\scrT$ satisfying the following conditions.

\begin{itemize}
    \item[$(1)$] For $\phi \in \bR$ and $0 \neq E \in \cP(\phi)$, we have $Z(v_X(E))\in \bR_{>0}\exp({\sqrt{-1}\pi \phi})$. 
    \item[$(2)$] For all $\phi\in\bR$, we have $\cP(\phi+1)=\cP(\phi)[1]$.
    \item[$(3)$] For $\phi_1>\phi_2$ and $E_i\in\cP(\phi_i)$ $(i=1,2)$, we have $\Hom(E_1,E_2)= 0$.  
\item[$(4)$] For any non-zero object $E \in \cT$, there exist real numbers 
\[\phi_1 > \cdots >\phi_n\]
and  a sequence 
\[
 0=E_0 \xrightarrow{f_1}E_1 \xrightarrow{f_2} E_2 \to \cdots \to E_{n-1} \xrightarrow{f_n} E_n=E 
 \]
of morphisms in $\cT$ such that the mapping cone $\Cone(f_i)$ is contained in $\cP(\phi_k)$ for $1 \leq i \leq n$.
\item[$(5)$] There is a positive real number $C$ such that 
for any $\phi \in \bR$ and a non-zero object $E \in \cP(\phi)$, we have 
\[ ||v_X(E)|| \leq C \cdot |Z(v_X(E))|. \]
\end{itemize}
Denote the set of all stability conditions on $\Db(X)$ by $\Stab(\Db(X))$.
\end{dfn}

By \cite[Theorem 1.2, Proposition 8.1]{bri07}, $\Stab(\Db(X))$ has the structure of a topological space.
We denote by $\Stab^\dagger(X)$ the distinguished connected component of $\Stab(\Db(X))$ in {\cite[Theorem 1.1]{bri08}}.
The notion of $\sigma$-(semi)stable objects is important for our purpose.

\begin{dfn}
Let $\sigma=(Z,\cP)$ be a stability condition on $\Db(X)$. For $\phi \in \bR$, a non-zero object $E \in \cP(\phi)$ is called a $\sigma$-{\it semistable object} of phase $\phi$. A simple object $E$ of $\cP(\phi)$ is called a $\sigma$-{\it stable object} of phase $\phi$. 
\end{dfn}

For a stability condition $\sigma=(Z,\cP) \in \Stab(\Db(X))$,
note that $\cP(\phi)$ is an abelian category for all $\phi \in \cP(\phi)$ by \cite[Lemma 5.2]{bri07}.

A stability condition $\sigma$ on $\Db(X)$ is {\it geometric} if there is a real number $\phi$ such that for any closed point $x \in X$, $\cO_x$ is $\sigma$-stable of phase $\phi$. Denote by $U(X)$  the set of geometric stability conditions on $\Db(X)$. By {\cite[Section 10, Section 11]{bri08}}, $U(X)$ is a non-empty open subset of $\Stab^*(X)$, and it is called the {\it geometric chamber} of $X$.

\subsection{Key results by Bayer and Bridgeland}~

In this subsection, we recall some results in \cite{bb}.
Let $X$ be a complex K3 surface.
For an autoequivalence $\Phi \colon \Db(X) \simto \Db(X)$, we have the {\it cohomological Fourier-Mukai transform}
\[\Phi^H \colon \Halg(X,\bZ) \simto \Halg(X,\bZ)\] of $\Phi$ with the following commutative diagram.
\[\xymatrix{K_0(X) \ar[d]_{v_X} \ar[r]^{\Phi^K} & K_0(X) \ar[d]^{x_X} \\
\Halg(X,\bZ) \ar[r]^{\Phi^H} &  \Halg(X,\bZ) } \]
Here, 
\[ \Phi^K \colon K_0(X) \simto K_0(X),~[E] \mapsto [\Phi(E)]\]
is the automorphism of $K_0(X)$. 
For the details, see \cite[Section 5.2]{huybre} and \cite[Lemma 10.6]{huybre}. 
 
\begin{dfn}[{\cite[Lemma 8.2]{bri07}}]
For an autoequivalence $\Phi \in \Aut \Db(X)$ and a stability condition $\sigma=(Z,\cP) \in \Stab(\Db(X))$, we define the stability condition 
\[ \Phi\sigma\defeq(Z \circ (\Phi^H)^{-1}, \Phi(\cP)) \in \Stab(\Db(X)),\]
where we put $\Phi(\cP)\defeq\{\Phi(\cP(\phi))\}_{\phi \in \bR}$.
\end{dfn}

By {\cite[Lemma 8.2]{bri07}}, the group $\Aut\Db(X)$ acts on $\Stab(\Db(X))$. It is not known that $\Aut\Db(X)$ preserves the connected component $\Stab^\dagger(X)$ in general. However, the following is known.

\begin{prop}[{\cite[Theorem 1.3]{bb}}]
Assume that the Picard number of $X$ is one. 
The group $\Aut\Db(X)$ preserves the connected component $\Stab^\dagger(X)$.
\end{prop}

We will use the following result about the geometric chamber $U(X)$.

\begin{lem}[{\cite[Lemma 6.7]{bb}}]\label{lem:geometric chamber} 
Assume that the Picard number of $X$ is one.
Take a stability condition $\sigma \in \Stab^\dagger(X)$. 
If there is a closed point $x \in X$ such that $\cO_x$ is $\sigma$-stable, we have $\sigma \in U(X)$.
\end{lem}

\section{Main results}
\subsection{Fourier-Mukai locus of a K3 surface of Picard number one}~

In this subsection, we prove Theorem \ref{thm:main 1}.
Let $X$ be a complex K3 surface. 
First, we prove that a birational autoequivalence of $\Db(X)$ is standard when the Picard number of $X$ is one. 
We need the following lemma.

\begin{lem}\label{lem:chamber and standard}
Take a stability condition $\sigma \in U(X)$.
Let $\Phi \colon \Db(X) \simto \Db(X)$ be an autoequivalence satisfying $\Phi\sigma \in U(X)$ and $\Phi^H(v)=v$, where $v\defeq(0,0,1)$ in $\Halg(X,\bZ)$.
Then we have $\Phi \in \Aut^{\mathrm{st}}\Db(X)$.
\end{lem}
\begin{proof}
Denote by $M_\sigma(v)$  the moduli space of $\sigma$-stable objects with Mukai vector $v$. Here, we fix a phase so that $\cO_x \in M_\sigma(v)$ holds for any closed point $x \in X$. 
For the details of $M_\sigma(v)$, see \cite{bm} for example.
Note that we have an isomorphism
\[ X \simto M_\sigma(v),~x \mapsto \cO_x. \]
By the assumption, we have the isomorphism
\[\Phi \colon M_\sigma(v) \simto M_{\Phi\sigma}(v),~E \mapsto \Phi(E).  \]
Here, we also fix a phase so that $\Phi(\cO_x) \in M_{\Phi\sigma}(v)$ holds for any closed point $x \in X$.
Since $\Phi \sigma$ is also geometric, composing even shifts to $\Phi$, we may assume that 
\[ M_{\Phi\sigma}(v)=M_\sigma(v).  \]
Therefore, for any closed point $x \in X$, there exists a closed point $y \in X$ such that $\Phi(\cO_x) \simeq \cO_y$.
By \cite[Corollary 5.23]{huybre}, there is a line bundle $\cL$ on $X$ and an automorphism $f$ of $X$ such that $\Phi \simeq f^* \circ (-) \otimes_X \cL$ holds.
\end{proof}

\begin{prop}\label{prop:birational=standard}
Assume that the Picard number of $X$ is one.
Consider an autoequivalence $\Phi \colon \Db(X) \simto \Db(X)$. If $\Phi$ is birational, we have
\[\Phi \in \Aut^{\mathrm{st}}\Db(X). \] 
\end{prop}
\begin{proof}
Since $\Phi$ is birational, there is a non-empty open subset $U$ of $X$ and an open immersion $\varphi \colon U \to X$ such that for any closed point $x \in U$, we have $\Phi(\cO_x) \simeq \cO_{\varphi(x)}$.

Take a geometric stability condition $\sigma \in U(X)$ and a closed point $x_0 \in U$.
Then $\Phi(\cO_{x_0})$ is $\Phi\sigma$-stable. 
By Remark \ref{rem:point to point}, $\cO_{\varphi(x_0)}$ is
$\Phi\sigma$-stable. By Lemma \ref{lem:geometric chamber}, we have $\Phi\sigma \in U(X)$.
By Lemma \ref{lem:chamber and standard}, we have $\Phi \in \Aut^{\mathrm{st}}\Db(X)$.
\end{proof}

The autoequivalence group $\Aut \Db(X)$ naturally acts on $\Eq(\Db(X), \scrT)$ from the right. We define the set $I_X$ as
\[ I_X\defeq\Eq(\Db(X), \scrT)/\Aut^{\mathrm{st}}\Db(X). \]

The following is the first main theorem of this paper.

\begin{thm}\label{thm:main A}
Let $\scrT$ be a triangulated category equivalent to the derived category of some complex K3 surface of Picard number one. 
The followings hold.
\begin{itemize}
\item[$(1)$] We have
\[ \Spec^{\FM}_{\triangle}\scrT=\bigsqcup_{X \in \FM(\scrT)}\Spec_{X}\scrT\]
\item[$(2)$] For $X \in \FM(\scrT)$, we have
\[ \Spec_{X}\scrT=\bigsqcup_{\Phi \in I_X}\Phi\left(\Spec_{\otimes}\Db(X) \right). \]
\end{itemize}
In particular, 
\[ \Spec^{\FM}_{\triangle}\scrT=\bigsqcup_{X \in \FM(\scrT)} \bigsqcup_{\Phi \in I_X}\Phi\left(\Spec_{\otimes}\Db(X) \right) \]
holds.
\end{thm}
\begin{proof}
First, we prove $(1)$.
Note that for K3 surfaces $X$ and $Y$, if $\Psi \colon \Db(X) \simeq \Db(Y)$ is a birational Fourier-Mukai transform, then $X$ and $Y$ are isomorphic. By Proposition \ref{prop:intersection},  we obtain
\[ \Spec^{\FM}_{\triangle}\scrT=\bigsqcup_{X \in \FM(\scrT)}\Spec_{X}\scrT. \]

Next, we prove $(2)$.
Take $X \in \FM(\scrT)$. For equivalences 
\[ \Phi_1 \colon \Db(X) \simto \scrT,~ \Phi_2 \colon \Db(X) \simto \scrT,\] 
\[ \Phi_1\left(\Spec_{\otimes}\Db(X) \right) \cap \Phi_2\left(\Spec_{\otimes}\Db(X) \right) \neq \emptyset \]
holds if and only if $\Phi^{-1}_2 \circ \Phi_1$ is birational up to shifts by Proposition \ref{prop:intersection}. By Proposition \ref{prop:birational=standard}, $\Phi^{-1}_2 \circ \Phi_1$ is birational up to shifts if and only if $\Phi^{-1}_2 \circ \Phi_1$ is a standard autoequivalence of $\Db(X)$.
Therefore, we obtain 
\[ \Spec_{X}\scrT=\bigsqcup_{\Phi \in I_X}\Phi\left(\Spec_{\otimes}\Db(X) \right). \]
\end{proof}

\subsection{Serre invariant locus and Fourier-Mukai locus}~

In this subsection, we prove that Conjecture \ref{conj:strong} and Conjecture \ref{conj:weak} are not true when a triangulated category $\scrT$ is equivalent to the derived category of some complex K3 surface of Picard number one.

Let $\scrT$ be a triangulated category equivalent to the derived category of some complex K3 surface.
An object $E$ of $\scrT$ is said to be {\it spherical} if 
\[ \RHom(E,E) \simeq \bC \oplus \bC[-2] \]
holds in $\Db(\Spec \bC)$. 

Let $X$ be a K3 surface.
For a spherical object $E \in \Db(X)$, we have
\[ \langle v_X(E), v_X(E) \rangle=-2. \]

\begin{lem}\label{lem:rank zero}
Assume that the Picard number of $X$ is one.
For an object $E \in \Db(X)$ with $\rk E=0$, $E$ is not spherical.
\end{lem}
\begin{proof}
Let $H$ be an ample generator of $\NS(X)$. There are integers $c$ and $m$ satisfying $v_X(E)=(0, cH, m)$. Then we have
\[ \langle v_X(E), v_X(E)  \rangle=c^2H^2 \geq 0.  \]
Since $\langle v_X(E), v_X(E)  \rangle \neq -2$, the object $E$ is not spherical.
\end{proof}

By Remark \ref{rem:torsion} and Lemma \ref{lem:rank zero}, we have the following remark. 

\begin{rem}\label{rem:do not have spherical}
Assume that the Picard number of $X$ is one.
For $x \in X$, the prime thick subcategory $\scrS_X(x)$ does not contain a spherical object.
\end{rem}

The following is deduced from Remark \ref{rem:do not have spherical}.

\begin{lem}\label{lem:spherical and FM locus}
Assume that the Picard number of $X$ is one.
For $\scrP \in \Spec^{\FM}_{\triangle}\scrT$, the prime thick subcategory $\scrP$ does not contain a spherical object.
\end{lem}

On the other hand, the following holds.

\begin{lem}\label{lem:spherical in maximal}
There is a maximal thick subcategory $\scrM$ of $\scrT$ such that $\scrM$ contains a spherical object.
\end{lem}
\begin{proof}
By the assumption, the triangulated category $\scrT$ is equivalent to the derived category of some K3 surface $X$. Since the structure sheaf $\cO_X$ of $X$ is a spherical object in $\Db(X)$, it is enough to show that there is a maximal thick subcategory $\scrM$ of $\Db(X)$ containing $\cO_X$. 
 By \cite[Proposition 1.1]{ll}, $\cO_X$ is not a split generator of $\Db(X)$. Therefore, there is a maximal thick subcategory containing $\cO_X$ by Proposition \ref{prop:non-empty}.
\end{proof}

We obtain the second main theorem of this paper.

\begin{thm}\label{thm:main B}
Let $\scrT$ be a triangulated category equivalent to the derived category of some K3 surface of Picard number one. 
Then there is a maximal thick subcategory $\scrM \in \Spec_{\triangle}\scrT$ such that $\scrM \notin \Spec^{\FM}_{\triangle}\scrT$ holds.  
\end{thm}
\begin{proof}
By Lemma \ref{lem:spherical in maximal}, there is a maximal thick subcategory $\scrM$ of $\scrT$ containing a spherical object of $\scrT$. Note that $\scrM$ is a prime thick subcategory of $\scrT$. By Lemma \ref{lem:spherical and FM locus}, we  have $\scrM \notin \Spec^{\FM}_{\triangle}\scrT$.
\end{proof}

By Remark \ref{rem:CY and Serre} and Theorem \ref{thm:main B}, we obtain a counterexample of Conjecture \ref{conj:strong} and Conjecture \ref{conj:weak}.

\end{document}